\documentclass{article}
\usepackage{graphicx,xcolor} 
\usepackage[a4paper, top=1in, bottom=1in, left=1in, right=1in]{geometry}
\usepackage[backend=biber,sorting=ynt]{biblatex}
\usepackage{style}
\newcount\Comments  
\newcommand{\kibitz}[2]{\ifnum\Comments=1\textcolor{#1}{#2} \fi \ignorespaces}

\title{A Non-Convex Separation Through an Alternative\\ Proof for the Supporting Hyperplane Theorem}
\author{
Ali Taherinassaj\thanks{Harvard University, Cambridge, MA USA}\\ali\_taherinassaj@g.harvard.edu\and
Yiling Chen\footnotemark[1]\\yiling@seas.harvard.edu}

\date{Oct 1, 2023}

\begin{document}

\maketitle

\begin{abstract}
    We present a concise proof for the supporting hyperplane theorem. We then observe that the proof not only establishes the supporting hyperplane theorem but also extends it to a hyperplane separation theorem for certain non-convex sets. The insight is that two sets are separable by a hyperplane if there is a point from whose perspective the two sets appear to be convex. This global topological to local vector space relationship anticipates generalizations to manifolds as a possible future direction. The result might also be of pedagogical interest as it re-wires the structure in which the supporting hyperplane theorem and its dependencies are usually presented in convex optimization and analysis textbooks. 
\end{abstract}

\section{Introduction}
    
    Many problems in machine learning, economics, and engineering often involve the use of optimization algorithms. Convex analysis is of particular theoretical interest because it provides a framework for analyzing the convergence rates of optimization algorithms, with the separating hyperplane theorem being central to this framework. Moreover, some algorithms, such as support vector machines \cite{vapnik}, directly utilize hyperplanes to classify data.
    
    Convex optimization textbooks, such as \cite{boyd_vandenberghe_2004}, traditionally follow a sequence where they first prove the separating hyperplane theorem and then establish the supporting hyperplane theorem and Farkas' lemma on top. However, a notable observation is that Farkas' lemma is a simpler theorem that can be proved linear algebraically by the Fourier–Motzkin Elimination (see \cite{ziegler}) and seems to precede separation theorems for arbitrary sets. In an effort to explore a more elementary basis, we aim to secure the theorem through a different, potentially more intuitive approach. 

    A slight re-organization reduces the separating hyperplane theorem to the supporting hyperplane theorem as done in \cite{Bertsekas2003ConvexAA} and they work with sequences and projection theorem to prove the supporting hyperplane theorem. We further reduce the supporting hyperplane theorem for arbitrary sets to the supporting hyperplane theorem for convex cones. Even though this small step makes no difference in the Euclidean space, it could provide the power to handle point-sets via their perspective cones\footnote{The two are translatable to each other via logarithmic map in Hadamard manifolds where the length-minimizing geodesics are unique.} that are subsets of the tangent space $T_pM$ of a point $p$ on a manifold $M$.

\section{Notation and Preliminaries}

Throughout the paper we denote the interior of a set $S$ by $\interior{(S)}$, its closure by $\cl{(S)}$, and its boundary by $\bd{(S)}$. For two sets $A$ and $B$, $A\setminus B$ denotes the set difference and $A\pm B$ are the Minkowski addition/subtraction.  We now define a few tools that we use in the main result.

\begin{definition}[Perspective Cone]\label{def:pcone}
    Let $S \subseteq \R^n$ be a set and $p \in \R^n$ be a point. Then the cone of $S$ from the perspective of $p$ is defined as the smallest cone with its tip on $p$ that encompasses $S$, i.e.
    \begin{align}
        \cone_p{(S)} := \{v \in \R^n \,|\, \exists\, t_v \in \R_{\geq 0}: t_v v \in S_p\}. 
    \end{align}
    where by $S_p$ we mean $S - p$. 
\end{definition}

\begin{remark}\label{rem:theRemark}
    When a point $p$ is a member of $S$, the perspective cone becomes trivial, i.e.
    \begin{align}
        \cone_p(S) = \R^n.
    \end{align}
\end{remark}

\section{Main Result}

\begin{theorem}[Supporting Hyperplane Theorem]\label{thm:support}
    Let $C \subseteq \R^n$ be a convex set and $p \not\in \interior{(C)}$, be a point outside the interior of $C$. There exists a hyperplane that passes through $p$ and contains $C$ in one of its half-spaces. 
\end{theorem}

    \begin{figure}[!h]
        \centering
        \includegraphics[scale=0.25]{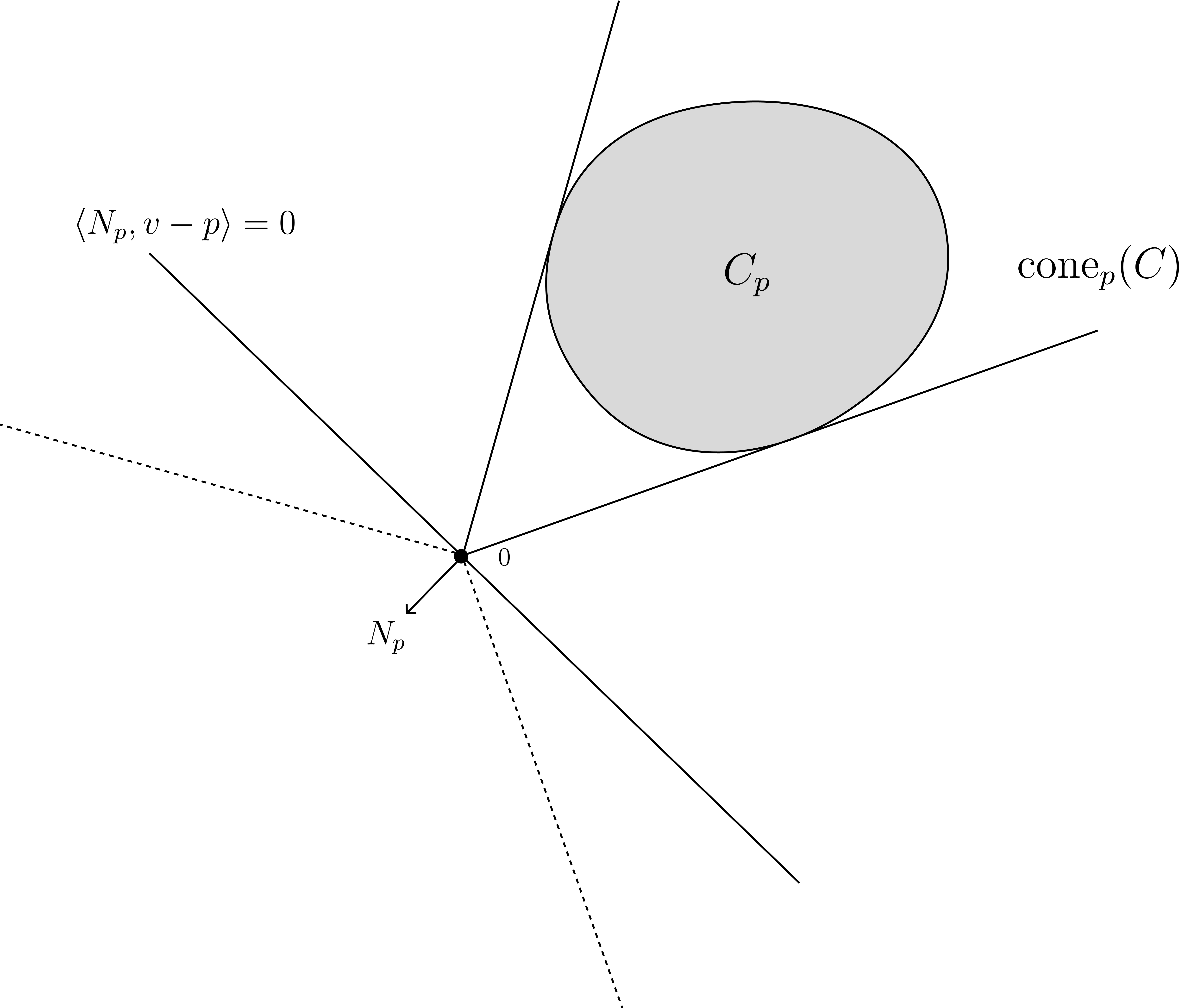}
        \caption{Illustration of the supporting hyperplane theorem (Thm. \ref{thm:support}).}
        \label{fig:support}
    \end{figure}

\begin{proof}

Assume $C$ is a proper and non-empty set\footnote{If $C$ were the whole space, there would be no point $p$ outside the interior of $C$ to begin with. Likewise, if $C$ were the empty set, it would always be a subset of any halfspace. In either case, the theorem would hold vacuously.} of $\R^n$. $C_p$ is a subset of $\cone_p(C)$ which itself is convex by Lem. \ref{lem:simple}; and $\cone_p{(C)}$ is a subset of a half-space because of its convexity by Lem. \ref{lem:mainLem}. 
    
\end{proof}

\begin{remark} A half-space passing through a point $p$ can be defined by a normal vector $N_p$:
    \begin{align}\label{eq:goodStuff}
        \exists N_p \in \R^n\setminus\{0\}: \forall v \in C: \langle N_p, v - p\rangle \leq 0.
    \end{align}
\end{remark}

\begin{remark}
    Notice that in the proof of Thm. \ref{thm:support} we only used the convexity of $C$ to imply the convexity of $\cone_p(C)$. If the $\cone_p(C)$ is convex without $C$ being convex, then Thm. \ref{thm:support} still holds. 
\end{remark}

\begin{theorem}[Generalized Separating Hyperplane Theorem]
    Let $A$ and $B$ be two arbitrary subsets of $\R^n$. If there is a point $p$ such that the two perspective cones $\cone_p(A)$ and $\cone_p(B)$ are convex and have disjoint interiors, then there is a hyperplane separating the two sets. 
\end{theorem}

\begin{figure}[!h]
    \centering
    \includegraphics[scale=0.25]{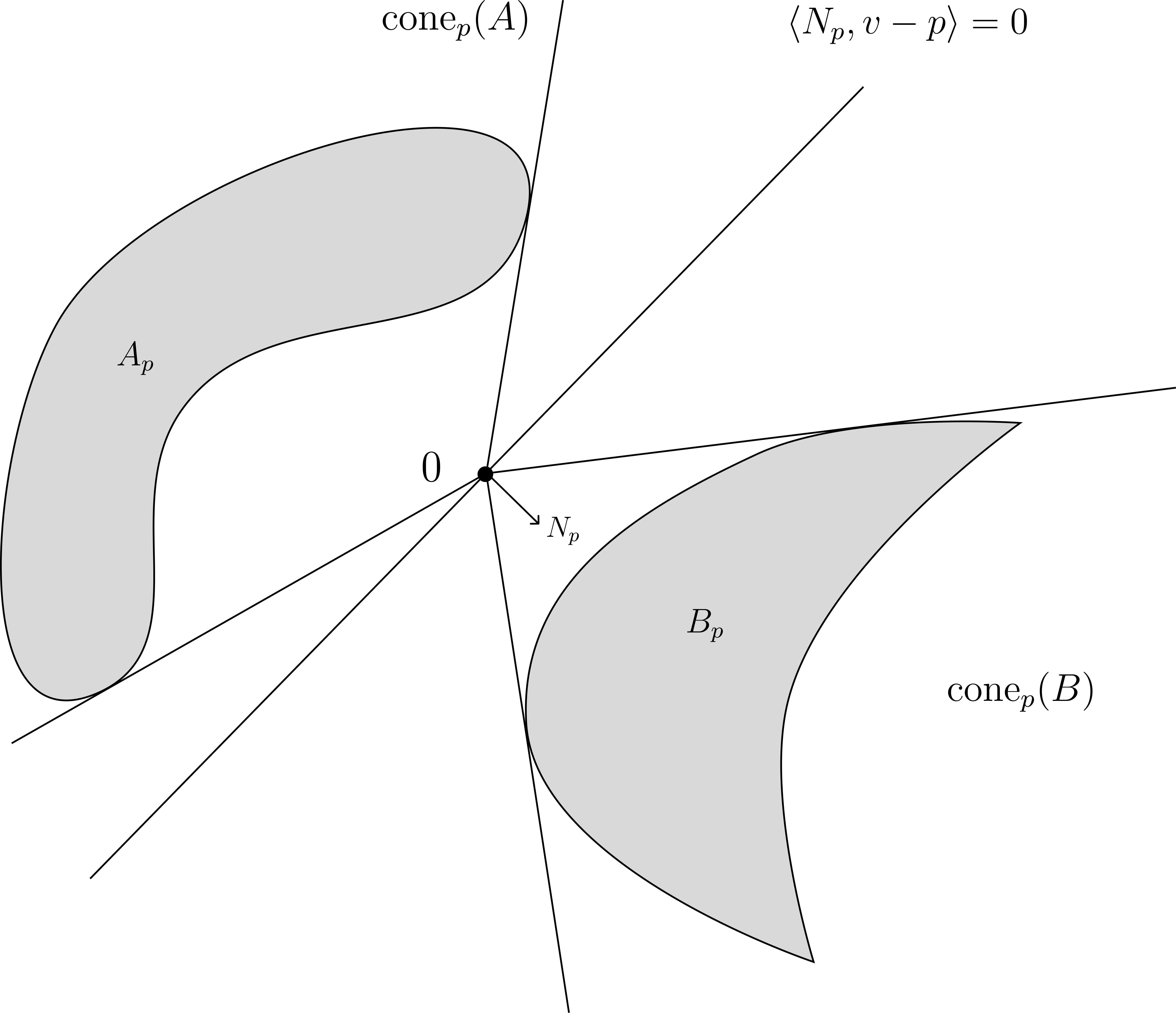}
    \caption{There could be many hyperplanes separating $A$ and $B$ as long as there is a point $p$ from whose perspective $A$ and $B$ appear to be convex.}
    \label{fig:separate-non-convex}
\end{figure}

\begin{proof}
    We can apply the supporting hyperplane theorem to  $C := \cone_p(A) - \cone_p(B)$ from the origin $p:=0$ as done in \cite{Bertsekas2003ConvexAA}. The set $C$ is convex because the Minkowski addition of two convex cones is convex. Moreover, the origin is on the boundary of $\cone_p(A) - \cone_p(B)$. This yields the existence of a hyperplane, i.e.
    \begin{align}
        \exists N_p \in \R^n \setminus \{0\}: \forall a \in A: \forall b \in B: \langle N_p, a\rangle \leq \langle N_p, b \rangle.
    \end{align}
\end{proof}

\section{Acknowledgments}

    A.T. would like to thank Nima Fasih for helpful conversations and their proofreading. This work was done as part of the course development for Harvard University's CS 128: Convex Optimization and Applications in Machine Learning taught by Dr. Yiling Chen. 

\nocite{*}
\printbibliography

\section{Appendix} 

Here we recall some basics of the cones and convexity. 

\begin{definition}[Cone]\label{def:cone}
    A set $K\subseteq \R^n$ is called a cone when it satisfies:
    \begin{align}
        \forall v\in K: \forall \theta \in \R_{> 0}: \theta v  \in K
    \end{align}
\end{definition}

\begin{lemma}[Convex Cone]\label{lem:convexcone}
    A cone $K\subseteq \R^n$ is convex if and only if:
    \begin{align}\label{eq:convexCones}
        \forall v, w \in K: v+w \in K
    \end{align}
\end{lemma}
\begin{proof}
    For the direct direction, if $v$ and $w$ are in the cone, then $2v$ and $2w$ are in the cone by Def. \ref{def:cone}. By convexity, the average of the two which is $v+w$ must thus be a member of the cone too.

    For the reverse direction we must show that for all two vectors $v$ and $w$, and for all $\theta \in [0, 1]$, the weighted average $\theta v + (1-\theta) w$ is in the cone as well given that the sum of any two vectors is in the cone. The cases of $\theta=0$ and $\theta \in (0, 1)$ must be treated separately. When $\theta=0$ the weighted combination happens to lie in the cone only by the virtue of the extra assumption that the vectors we began with were in the cone. It is for the other case that we can scale $v$ by $\theta > 0$ and $w$ by $(1-\theta)$ and still have the two vectors be in the cone to be able to sum them into $\theta v + (1-\theta) w$. 
\end{proof}

\begin{definition}[Conic Hull]
    The conic hull of a set $S\subseteq \R^n$ defined as:
    \begin{align}
        \chull{(S)} := \{\theta_v \cdot v + \theta_w \cdot w \,|\, v, w \in S \text{ and } \theta_v, \theta_w \in \R_{>0}\}
    \end{align}
    
    is the smallest convex cone that contains $S$. 
\end{definition}

\begin{lemma}\label{lem:simple}

Let $C \subseteq \R^n$ be a set. Then, convexity of $C$ implies that $\cone_p{(C)}$ is convex for any $p \in \R^n$.
\end{lemma}

\begin{proof}
    When $p\in C$, we are done by Rem. \ref{rem:theRemark}. For the remainder, the proof is by contraposition. We will show that the non-convexity of $\cone_p{(C)}$ implies the non-convexity of $C$. From the non-convexity of $\cone_p{(C)}$, we know there are two distinct points $v, w \in \cone_p{(C)}$ and some $\theta \in (0, 1)$ such that $\theta v + (1-\theta) w \notin \cone_p{(C)}$. \\ \\
    What $v \in \cone_p{(C)}$ means is that $p+t_{v} v\in C$ for some $t_{v} > 0$. Similarly, $w \in \cone_p{(C)}$ means $p+t_{w} {w} \in C$ for some $t_{w} > 0$. Now let $\Tilde{v}=t_{v} v$ and $\Tilde{w}=t_{w} w$. Now, $\Tilde{v}, \Tilde{w}\in \cone_p{(C)}$. By setting $\Tilde{\theta} = \dfrac{\theta t_w}{\theta t_w + (1-\theta)t_v}$, we can construct a convex combination of $\Tilde{v}$ and $\Tilde{w}$: $$\Tilde{\theta} \Tilde{v} + (1- \Tilde{\theta}) \Tilde{w},$$
    which equals $\dfrac{t_v t_w}{\theta t_w + (1-\theta)t_v} (\theta v + (1-\theta)w)$. Since $\theta v + (1-\theta)w \notin \cone_p{(C)}$ and $\cone_p{(C)}$ is a cone, we know $\Tilde{\theta} \Tilde{v} + (1- \Tilde{\theta}) \Tilde{w} \notin \cone_p{(C)}$. This means for any $t>0$, $$p + t\,(\Tilde{\theta} \Tilde{v} + (1- \Tilde{\theta}) \Tilde{w}) \notin C.$$
    For $t=1$, the above is equivalent to 
    $$\Tilde{\theta} (p + \Tilde{v}) + (1-\Tilde{\theta})(p+\Tilde{w}) \notin C.$$ But both $p+\Tilde{v}$ and $p+\Tilde{w}$ are in $C$, because $\Tilde{v}, \Tilde{w} \in \cone_p{(C)}$ (as $v, w \in \cone_p{(C)}$). This implies the non-convexity of $C$.
\end{proof}

\begin{lemma}\label{lem:interior}
    The interior of the complement of a proper convex set $C \subset \R^n$ is non-empty. 
\end{lemma}

\begin{proof}
    Suppose that $C^c$ has an empty interior. For any point $v \in C^c$, there must exist a $w \in \R^n$ such that $v +w \in C$ and $v -w \in C$. If there was no such $w$, it would have meant that for all vectors $w$ around including open balls of arbitrary radii $\epsilon$ around $v$ (i.e. $\operatorname{ball}_v(\epsilon):=\{w \mid\|w\| \leq \epsilon\}$) we had $v+w \in C^c$ which would contradict our assumption of empty interior for $C^c$.

    Now that there is such a vector $w$ such that $v+w \in C$ and $v-w \in C$, it is easy to see that the convexity of $C$ is violated because $v = \frac{1}{2}(v+w) + \frac{1}{2}(v-w)$ is not in $C$. Therefore it is impossible for $C^c$ to have an empty interior. 
\end{proof}

\begin{lemma}\label{lem:mainLem}
    Every proper\footnote{A cone that is a proper subset of $\R^n$.} convex cone $K \neq \R^n$ is a subset of a half-space. 
\end{lemma}

\begin{proof} 
    
    By Lem. \ref{lem:interior} we know $K^c$ has a non-empty interior. If there is a vector $v \in K^c$, such that  $-v \in K^c$ as well, then it is possible to take the conic hull of $\{v\}$ and $K$. After this operation, $-v$ cannot belong to \(\chull(\{v\} \cup K)\). Therefore, a convex cone $K$ can always be expanded unless there are no vectors $v \in K^c$ such that $-v \in K^c$.

    The fact that for every vector in $K^c$ we have its opposite inside $K$ makes $-K^c$ identical to $\interior{(K)}$ which means $K^c$ is now itself a convex cone. As a result, the boundary $\cl{(K)} \cap \cl{(K^c)}$ is the intersection of two convex cones and is therefore convex. 
    
    The boundary or the set of points $\{v \in K| -v \in K\}$ is closed under scalar action on top of being closed under addition by virtue of being a convex cone. This means that the boundary is an $(n-1)$-dimensional vector subspace, i.e. a hyperplane. 
    
    

\end{proof}

\end{document}